\documentclass[review,10]{elsarticle}

\usepackage{lipsum}
\usepackage{hyperref}
\usepackage[hyperref]{xcolor}

\hypersetup{
  colorlinks=true,
}
\usepackage[center]{caption}
\usepackage{xcolor}
\usepackage{mathdots}
\usepackage{amsmath}
\usepackage{amsfonts}
\usepackage{graphicx}
\usepackage{subfigure}
\usepackage{booktabs}
\usepackage{amsthm}
\usepackage{empheq}
\usepackage{amssymb}
\usepackage{mathrsfs}
\usepackage[
singlelinecheck=false 
]{caption}
\newtheorem{theorem}{Theorem}[section]
\newtheorem{lemma}{Lemma}[section]
\newtheorem{remark}{Remark}[section]

\newtheorem{prop}{Proposition}[section]

\usepackage{geometry}
\numberwithin{equation}{section}

\biboptions{authoryear}

\begin{document}

\begin{frontmatter}

\title{The first passage time on the (reflected) Brownian motion with broken drift hitting a random boundary}


\author[mymainaddress]{Zhenwen Zhao}
\ead{zzw@mail.nankai.edu.cn}

\author[mymainaddress]{Yuejuan Xi\corref{mycorrespondingauthor}}
\cortext[mycorrespondingauthor]{Corresponding author}
\ead{yjx@mail.nankai.edu.cn}

\address[mymainaddress]{School of Mathematical Sciences, Nankai University, Tianjin, PR China, 300071}

\begin{abstract}
  In this paper we consider a (reflected) Brownian motion with broken drift hitting a random boundary. Some dedicated calculations allow us to obtain the formula on the joint Laplace transform of the hitting time and hitting position. These develop the research on first rendezvous times of (reflected) Brownian motion and compound Poisson-type processes by \citet{MR2122800}.
\end{abstract}


\end{frontmatter}

\section{Introduction}
The reflected diffusion processes as a class of typical Markov processes have been widely applied in the subjects, such as queueing system, workload process, operation research and the others. Some references refer the reader to \citet{MR1092214}, \citet{MR2396593}, \citet{MR2122800}, etc. On the other hand, the threshold process with piecewise drift term or diffusion term has also been another interesting class of processes (see, \citet{MR3414915}, \citet{MR717388},  \citet{MR3998245}). Some specific threshold diffusion processes, for instance Brownian motion with alternative (broken) drift and Oscillating Brownian motion, were studied in \citet{MR1912205} and \citet{MR474526}.

The hitting boundary problem of a Markov process is a classical subject in the Markov process potential analysis. In recent years, most literature on the boundary crossing probability focuses on the linear or piecewise linear boundary (cf. \citet{MR1816120}, \citet{MR1429054}). For a Markov process hitting a random boundary, \citet{MR2122800} formulated some results of  the first rendezvous time of (reflected) Brownian motion (BM) and compound Poisson-type boundaries. \citet{MR2884811} extended the results  in \citet{MR2122800} to the (reflected) Ornstein-Uhlenbeck process. \citet{MR3102489} discussed the problems of some boundary crossing probabilities for a Brownian motion with different  stochastic boundaries, in particular, including compound Poisson process boundaries. In this paper, we consider a reflected Brownian motion (RBM) with broken drift hitting a random boundary. We will develop the above results in this direction.

Let us begin with a reflected diffusion process. Assume that  $(\Omega,\mathcal{F},\mathcal{F}_t,\mathcal{P})$ is a probability space satisfying the usual conditions, and $W=\{W_t, t\geq 0\}$ is a standard Brownian motion. Define the diffusion processes $X=\{X_t, t\geq 0\}$ and $\tilde{X}=\{\tilde{X}_t, t\geq 0\}$ via the following stochastic differential equations (SDEs):
\begin{align}\label{e:X}
dX_t&=\mu(X_t)dt+dW_t,\\\label{e:X1}
d\tilde{X_t}&=\mu(\tilde{X_t})dt+dW_t+dL_t,
\end{align}
where $\mu(x)=\mu_1\mathbf{1}_{\{x<c\}}+\mu_2\mathbf{1}_{\{x\geq c\}}$, $\mathbf{1}_{\{\cdot\}}$ is an indicator function,
$\mu_1$, $\mu_2\in\mathbb{R}$ are constants and $c$ is a positve constant. The process $L=\{L_t:~t\geq 0\}$ is the local time process of $X$ at the boundary zero and $L_t$ satisfies the property that
\begin{equation}\label{e:Lt}
L_t=\int_0^t \mathbf{1}_{\{\tilde{X}_s=0\}}dL_s~~~\textup{for all}~t\geq 0,
\end{equation}
please see \citet{MR2134100}, \citet{MR1092214} for details. We call $X$ and $\tilde{X}$ as Brownian motion with broken drift  and reflected Brownian motion with broken drift respectively. The existence and uniqueness of the SDEs \eqref{e:X} and \eqref{e:X1} were proved in \citet{pilipenko2014introduction}.

Next we introduce a random boundary $C(t)$ by
\begin{equation*}\label{e:C1}
C(t)\equiv b+Y \mathbf{1}_{\{T_1\leq t\}},
\end{equation*}
 where $b$ is a constant, $T_1$ admits an exponential law with parameter $\lambda>0$, which is independent of the Brownian motion W. In addition, the random variable $Y$ is independent of $(T,W)$, which admits a distribution function given by $F(dy)$.

Define the hitting times:
\begin{align}\label{e:tau}
\tau&=\inf\{t\geq 0: X_t=0~~\textup{or}~~X_t=C(t)\},\\\label{e:tau2}
\tilde{\tau}&=\inf\{t\geq 0: \tilde{X}_t=C(t)\},
\end{align}
by convention, $\inf{\O}=\infty$.

In \citet{MR2122800}, they computed first rendezvous times in the case of the  (reflected) BM. But for our case, due to the existence of the broken drift, using the conventional Meyer-Tanaka formula to construct martingale does not work. So we shall use time-dependent Meyer-Tanaka to derive first passage time of (reflected) BM hitting a random boundary in this paper.

The paper is organized as follows. In Section \ref{section2}, we present and calculate some preliminary results of BM with broken drift and RBM with broken drift, respectively. In Section \ref{section3}, we deal with the first passage time of a (reflected) Brownian motion with broken drift hitting a random boundary.
\section{Preliminaries}\label{section2}
In this section, we introduce some elementaries about the (reflected) BM with broken drift.
\subsection{ The BM with broken drift }
Assume $X$ is the diffusion process defined in \eqref{e:X}, then its infinitesimal generator defined on $C_b^2(\mathcal{R})$ should be
\begin{equation*}
\mathscr{A}f(x)=\frac{1}{2}f^{\prime\prime}(x)+\mu(x) f^{\prime}(x),~~~~~~x\neq c.
\end{equation*}
It follows from \citet[Page 128]{MR1912205} that the speed measure of $X$ is
\begin{equation*}
m(dx)=(2e^{2\mu_1x}\mathbf{1}_{\{x<c\}}+2e^{2(\mu_1-\mu_2)c}e^{2\mu_2x}\mathbf{1}_{\{x\geq c\}})dx
\end{equation*}
and the scale density $s(x)$ is
\begin{equation*}
s(x)=(e^{-2\mu_1x}\mathbf{1}_{\{x<c\}}+e^{2(\mu_2-\mu_1)c}e^{-2\mu_2x}\mathbf{1}_{\{x\geq c\}})dx.
\end{equation*}
Define
\begin{eqnarray*}
\psi_{s}(x)=
\begin{cases}
e^{\lambda_1^{+}x}, &x<c \cr A_1e^{\lambda_2^{-}x}+A_2e^{\lambda_2^{+}x}, &x\geq c
\end{cases}
\end{eqnarray*}
and
\begin{eqnarray*}
\varphi_{s}(x)=
\begin{cases}
B_1e^{\lambda_1^{-}x}+B_2e^{\lambda_1^{+}x}, &x<c,\cr e^{\lambda_2^{-}x}, &x\geq c,
\end{cases}
\end{eqnarray*}
where the coefficients $\lambda_1^{+}$, $\lambda_1^{-}$, $\lambda_2^{+}$, $\lambda_2^{-}$, $A_1$, $A_2$, $B_1$ and $B_2$ are defined respectively by:
\begin{equation*}
\lambda_1^{+}=-\mu_1+\sqrt{\mu_1^2+2s},~~~\lambda_1^{-}=-\mu_1-\sqrt{\mu_1^2+2s},
\end{equation*}
\begin{equation*}
\lambda_2^{+}=-\mu_2+\sqrt{\mu_2^2+2s},~~~\lambda_2^{-}=-\mu_2-\sqrt{\mu_2^2+2s},
\end{equation*}
\begin{equation*}
A_1=\frac{\lambda_2^{+}-\lambda_1^{+}}{\lambda_2^{+}-\lambda_2^{-}}e^{(\lambda_1^{+}-\lambda_2^{-})c},~~~ A_2=\frac{\lambda_1^{+}-\lambda_2^{-}}{\lambda_2^{+}-\lambda_2^{-}}e^{(\lambda_1^{+}-\lambda_2^{+})c},
\end{equation*}
\begin{equation*}
B_1=\frac{\lambda_1^{+}-\lambda_2^{-}}{\lambda_1^{+}-\lambda_1^{-}}e^{(\lambda_2^{-}-\lambda_1^{-})c},~~~ B_2=\frac{\lambda_2^{-}-\lambda_1^{-}}{\lambda_1^{+}-\lambda_1^{-}}e^{(\lambda_2^{-}-\lambda_1^{+})c}.
\end{equation*}
Then we can check $\psi_{s}(.)$ and $\varphi_{s}(.)$ are the unique (up to a multiplicative constant) increasing and decreasing solutions of the following ordinary differential equation (ODE)
\begin{equation}\label{e:ode1}
\mathcal{A}f(x)=\frac{1}{2}f^{\prime\prime}(x)+\mu(x)f^{\prime}(x)=s f(x), ~x\neq c,
\end{equation}
which satisfy the boundary conditions
$\mathcal{D}(\mathcal{A})=\{f:f,\mathcal{L}f\in \mathcal{C}_b(\mathcal{R}),f\in  \mathcal{C}^{1}_b(\mathcal{R})\}$ (see the Theorem VII.3.12 in \citet{MR1725357} for details).

\par
Let $p(t,x,y)$ be the transition density function of $X$, and define $G_s(x,y)$ by
\begin{equation*}
G_{s}(x,y)=\int_0^{\infty} e^{-s t} p(t;x,y) dt.
\end{equation*}
Then by Proposition 11 \citet[Section 2.1]{MR1912205}, $G_{s}(x,y)$ can be formulated as
\begin{equation}\label{e:Gs}
G_{s}(x,y)=\frac{m(y)}{\omega_{s}}\psi_{s}(x\wedge y)\varphi_{s}(x\vee y),
\end{equation}
where the Wronskian $\omega_{s}$ is a constant independent of $x$ and
\begin{equation}\label{e:ws}
\omega_{s}=\frac{1}{s(x)}(\varphi_{s}(x)\psi^\prime_{s}(x)-\varphi^\prime_{s}(x)\psi_{s}(x)).
\end{equation}
Next we can get the exact expression of the Green function $G_{s}(x,y)$.
\begin{prop}
(1) for $x<c$,
\par
\begin{eqnarray*}
G_{s}(x,y)=
\begin{cases}
\frac{2e^{2\mu_1 y}}{\omega_{s}}(B_1e^{\lambda_1^{-}x}+B_2e^{\lambda_1^{+}x})e^{\lambda_1^{+}y},&y\leq x, \cr \frac{2e^{2\mu_1 y}}{\omega_{s}}(B_1e^{\lambda_1^{-}y}+B_2e^{\lambda_1^{+}y})e^{\lambda_1^{+}x}, &x<y<c, \cr \frac{2e^{2(\mu_1-\mu_2)c}e^{2\mu_2 y}}{\omega_{s}}e^{\lambda_1^{+}x}e^{\lambda_2^{-}y},  &y\geq c,
\end{cases}
\end{eqnarray*}
\par
(2) for $x\geq c$,
\begin{eqnarray*}
G_{s}(x,y)=
\begin{cases}
\frac{2e^{2\mu_1 y}}{\omega_{s}} e^{\lambda_1^{+}y}e^{\lambda_2^{-}x},&y<c, \cr \frac{2e^{2(\mu_1-\mu_2)c}e^{2\mu_2 y}}{\omega_{s}}(A_1e^{\lambda_2^{-}y}+A_2e^{\lambda_2^{+}y})e^{\lambda_2^{-}x}, &c\leq y\leq x, \cr \frac{2e^{2(\mu_1-\mu_2)c}e^{2\mu_2 y}}{\omega_{s}}(A_1e^{\lambda_2^{-}x}+A_2e^{\lambda_2^{+}x})e^{\lambda_2^{-}y},  &y>x,
\end{cases}
\end{eqnarray*}
where the Wronskian is
\begin{equation*}
\omega_{s}=(\lambda_1^{+}-\lambda_2^{-})e^{(\lambda_2^{-}-\lambda_1^{-})c}.
\end{equation*}
\end{prop}
Thus the transition density function $p(t;x,y)$ of $X$ can follows from inverse Laplace transform as
\begin{equation}\label{e:p}
p(t;x,y)=\mathcal{L}^{-1}(G_{s}(x,y))(t).
\end{equation}
This will be used in our derivations below.
\subsection{RBM with broken drift }
\par
Recall $\tilde{X}$, a reflected Brownian motion $\tilde{X}$ with broken drift, defined in \eqref{e:X1}, and its infinitesimal generator for $f\in C_b^2(\mathcal{R})$ should be
\begin{equation*}
\mathscr{A}f(x)=\frac{1}{2}f^{\prime\prime}(x)+\mu(x) f^{\prime}(x),~~~~~~x>0~\textup{and}~x\neq c.
\end{equation*}
The speed measure of $\tilde{X}$ is
\begin{equation*}
m(dx)=(2e^{2\mu_1x}\mathbf{1}_{\{0\leq x<c\}}+2e^{2(\mu_1-\mu_2)c}e^{2\mu_2x}\mathbf{1}_{\{x\geq c\}})dx
\end{equation*}
and the scale density $\tilde{s}(x)$ is
\begin{equation*}
\tilde s(x)=(e^{-2\mu_1x}\mathbf{1}_{\{0<x<c\}}+e^{2(\mu_2-\mu_1)c}e^{-2\mu_2x}\mathbf{1}_{\{x\geq c\}})dx.
\end{equation*}
Set
\begin{eqnarray}\label{e:tilde(psai)}
\tilde{\psi}_{s}(x)=
\begin{cases}
\lambda_1^{+}e^{\lambda_1^{-}x}-\lambda_1^{-}e^{\lambda_1^{+}x}, &0\leq x<c ,\cr \tilde{A}_1e^{\lambda_2^{-}x}+\tilde{A}_2e^{\lambda_2^{+}x}, &x\geq c,
\end{cases}
\end{eqnarray}
and
\begin{eqnarray*}
\tilde{\varphi}_{s}(x)=
\begin{cases}
\tilde{B}_1e^{\lambda_1^{-}x}+\tilde{B}_2e^{\lambda_1^{+}x}, &0\leq x<c,\cr e^{\lambda_2^{-}x}, &x\geq c,
\end{cases}
\end{eqnarray*}
where the coefficients $\lambda_1^{+}$, $\lambda_1^{-}$, $\lambda_2^{+}$, $\lambda_2^{-}$, $\tilde{A}_1$, $\tilde{A}_2$, $\tilde{B}_1$ and $\tilde{B}_2$ are explicitly expressed by:
\begin{equation*}
\lambda_1^{+}=-\mu_1+\sqrt{\mu_1^2+2s},~~~\lambda_1^{-}=-\mu_1-\sqrt{\mu_1^2+2s},
\end{equation*}
\begin{equation*}
\lambda_2^{+}=-\mu_2+\sqrt{\mu_2^2+2s},~~~\lambda_2^{-}=-\mu_2-\sqrt{\mu_2^2+2s},
\end{equation*}
\begin{equation*}
\tilde{A}_1=\lambda_1^{+}\frac{\lambda_2^{+}-\lambda_1^{-}}{\lambda_2^{+}-\lambda_2^{-}}e^{(\lambda_1^{-}-\lambda_2^{-})c}-
\lambda_1^{-}\frac{\lambda_2^{+}-\lambda_1^{+}}{\lambda_2^{+}-\lambda_2^{-}}e^{(\lambda_1^{+}-\lambda_2^{-})c},
\end{equation*}

\begin{equation*}
\tilde{A}_2=\lambda_1^{+}\frac{\lambda_1^{-}-\lambda_2^{-}}{\lambda_2^{+}-\lambda_2^{-}}e^{(\lambda_1^{-}-\lambda_2^{+})c}-
\lambda_1^{-}\frac{\lambda_1^{+}-\lambda_2^{-}}{\lambda_2^{+}-\lambda_2^{-}}e^{(\lambda_1^{+}-\lambda_2^{+})c},
\end{equation*}

\begin{equation*}
\tilde{B}_1=\frac{\lambda_1^{+}-\lambda_2^{-}}{\lambda_1^{+}-\lambda_1^{-}}e^{(\lambda_2^{-}-\lambda_1^{-})c},~~~ \tilde{B}_2=\frac{\lambda_2^{-}-\lambda_1^{-}}{\lambda_1^{+}-\lambda_1^{-}}e^{(\lambda_2^{-}-\lambda_1^{+})c}.
\end{equation*}
Then we can check $\tilde{\psi}_{s}(.)$ and $\tilde{\varphi}_{s}(.)$ are the unique (up to a multiplicative constant) increasing and decreasing solutions of the following ordinary differential equation (ODE)
\begin{equation*}
\mathcal{L}f(x)=\frac{1}{2}f^{\prime\prime}(x)+\mu(x)f^{\prime}(x)=s f(x), ~x>0~\textup{and}~x\neq c,
\end{equation*}
 which satisfy the boundary conditions $\mathcal{D}(\mathcal{L})=\{f:f,\mathcal{L}f\in \mathcal{C}_b([0,+\infty)),~f^{\prime}(0+)=0,~ f^{\prime}(c-)=f^{\prime}(c+)\}$.

Similarly, from \eqref{e:Gs} and \eqref{e:ws} we can get the Green function $\tilde{G}_{s}(x,y)$.
\begin{prop}
(1) for $0<x<c$,
\begin{eqnarray*}
\tilde{G}_{s}(x,y)=
\begin{cases}
\frac{2e^{2\mu_1 y}}{\tilde{\omega}_{s}}(\tilde{B}_1e^{\lambda_1^{-}x}+\tilde{B}_2e^{\lambda_1^{+}x})(\lambda_1^{+}e^{\lambda_1^{-}y}-\lambda_1^{-}e^{\lambda_1^{+}y}),&0<y\leq x, \cr \frac{2e^{2\mu_1 y}}{\tilde{\omega}_{s}}(\tilde{B}_1e^{\lambda_1^{-}y}+\tilde{B}_2e^{\lambda_1^{+}y})(\lambda_1^{+}e^{\lambda_1^{-}x}-\lambda_1^{-}e^{\lambda_1^{+}x}), &x<y<c, \cr \frac{2e^{2(\mu_1-\mu_2)c}e^{2\mu_2 y}}{\tilde{\omega}_{s}}(\lambda_1^{+}e^{\lambda_1^{-}x}-\lambda_1^{-}e^{\lambda_1^{+}x})e^{\lambda_2^{-}y},  &y\geq c,
\end{cases}
\end{eqnarray*}
\par
(2) for $x\geq c$,
\begin{eqnarray*}
\tilde{G}_{s}(x,y)=
\begin{cases}
\frac{2e^{2\mu_1 y}}{\tilde{\omega}_{s}} (\lambda_1^{+}e^{\lambda_1^{-}y}-\lambda_1^{-}e^{\lambda_1^{+}y})e^{\lambda_2^{-}x},&0<y<c, \cr \frac{2e^{2(\mu_1-\mu_2)c}e^{2\mu_2 y}}{\tilde{\omega}_{s}}(\tilde{A}_1e^{\lambda_2^{-}y}+\tilde{A}_2e^{\lambda_2^{+}y})e^{\lambda_2^{-}x}, &c\leq y\leq x, \cr \frac{2e^{2(\mu_1-\mu_2)c}e^{2\mu_2 y}}{\tilde{\omega}_{s}}(\tilde{A}_1e^{\lambda_2^{-}x}+\tilde{A}_2e^{\lambda_2^{+}x})e^{\lambda_2^{-}y},  &y>x,
\end{cases}
\end{eqnarray*}
where the Wronskian is
\begin{equation*}
\tilde{\omega}_{s}=-\lambda_1^{+}(\lambda_2^{-}-\lambda_1^{-})e^{(\lambda_2^{-}-\lambda_1^{+})c}-
\lambda_1^{-}(\lambda_1^{+}-\lambda_2^{-})e^{(\lambda_2^{-}-\lambda_1^{-})c}.
\end{equation*}

\end{prop}
Thus the transition density function $\tilde{p}(t;x,y)$ of $\tilde{X}$ is denoted by
\begin{equation}\label{e:tp}
\tilde{p}(t;x,y)=\mathcal{L}^{-1}(\tilde{G}_{s}(x,y))(t).
\end{equation}

\section{Hitting problems of BM with broken drift and RBM with broken drift}\label{section3}
\par
\subsection{Hitting problem of BM with broken drift}
In this subsection, we shall study the first passage time $\tau$ defined in \eqref{e:tau}. As in \citet{MR2122800}, we will compute the joint Laplace transform of the random vector $(X_{\tau},\tau)$, i.e.,
\begin{equation*}
\Psi(\alpha,\theta;x)=\mathbf{E}_x\left[e^{-\alpha X_{\tau}-\theta\tau}\right],
\end{equation*}
where $\mathbf{E}_x\left[\cdot\right]=\mathbf{E}\left[\cdot~|~X_0=x\right]$. For this purpose, we are first to compute the following four Laplace transforms:
\begin{align*}
\Phi_1(\alpha,\theta;x)&=\mathbf{E}_x\left[e^{-\alpha X_{T_1}-\theta T_1}\mathbf{1}_{\{X_{T_1<c}\}}\right],\\
\Phi_2(\alpha,\theta;x)&=\mathbf{E}_x\left[e^{-\alpha X_{T_1}-\theta T_1}\mathbf{1}_{\{X_{T_1\geq c}\}}\right],\\
\Phi_3(\alpha,\theta;x)&=\mathbf{E}_x\left[e^{-\alpha X_{T_1}-\theta T_1}\mathbf{1}_{\{X_{T_1< c}\}}\mathbf{1}_{\{\tau<T_1\}}\right],\\
\Phi_4(\alpha,\theta;x)&=\mathbf{E}_x\left[e^{-\alpha X_{T_1}-\theta T_1}\mathbf{1}_{\{X_{T_1\geq c}\}}\mathbf{1}_{\{\tau<T_1\}}\right].
\end{align*}
On the other hand, we need these two identities:
\begin{align}\label{e:w1}
\omega_1(\theta;b,x)&=\mathbf{E}_x\left[e^{-\theta R_0}\mathbf{1}_{\{R_0<R_b\}}\right],\\\label{e:w2}
\omega_2(\theta;b,x)&=\mathbf{E}_x\left[e^{-\theta R_b}\mathbf{1}_{\{R_0>R_b\}}\right],
\end{align}
where $R_a$ denotes the first passage time of the stochastic process $X=\{X_t: t\geq0\}$ over a constant boundary, i.e., $R_a=\inf{\{t\geq 0: X_t=a\}}$ for $a\in \mathcal{R}^1$.

\par
If $0<x<b$, since the function $\psi_{\theta}(x)$ and $\varphi_{\theta}(x)$ are the solutions to the ordinary differential equation (ODE) \eqref{e:ode1}, applying It\^o's formula and the optional stopping theorem yields:
\begin{align*}
\mathbf{E}_x\left[e^{-\theta (R_0\wedge R_b)}\psi_{\theta}(X_{R_0\wedge R_b})\right]=\psi_{\theta}(x),\\
\mathbf{E}_x\left[e^{-\theta (R_0\wedge R_b)}\varphi_{\theta}(X_{R_0\wedge R_b})\right]=\varphi_{\theta}(x).
\end{align*}
So we can obtain
\begin{align}\label{2.1}
\psi_{\theta}(x)&=\psi_{\theta}(0)\omega_1(\theta;b,x)+\psi_{\theta}(b)\omega_2(\theta;b,x),\\\label{2.2}
\varphi_{\theta}(x)&=\varphi_{\theta}(0)\omega_1(\theta;b,x)+\varphi_{\theta}(b)\omega_2(\theta;b,x).
\end{align}
Solving $\omega_1(\theta;b,x)$, $\omega_2(\theta;b,x)$ in \eqref{2.1} and \eqref{2.2}, we yield the expression of $\omega_1(\theta;b,x)$, $\omega_2(\theta;b,x)$:

\begin{equation}\label{omega1}
\omega_1(\theta;b,x)=\frac{\varphi_{\theta}(x)\psi_{\theta}(b)-\psi_{\theta}(x)\varphi_{\theta}(b)}
                        {\varphi_{\theta}(0)\psi_{\theta}(b)-\psi_{\theta}(0)\varphi_{\theta}(b)}
\end{equation}
and
\begin{equation}\label{omega2}
\omega_2(\theta;b,x)=\frac{\varphi_{\theta}(x)\psi_{\theta}(0)-\psi_{\theta}(x)\varphi_{\theta}(0)}
                        {\varphi_{\theta}(b)\psi_{\theta}(0)-\psi_{\theta}(b)\varphi_{\theta}(0)}.
\end{equation}

We introduce the following time-dependent Meyer-Tanaka formula derived in \citet{MR2409002}.

\begin{lemma}\label{le:4}
If $f(t,x)$ satisfies the following conditions.
\par
(1) $f$ is absolutely continuous in each variable.
\par
(2) $\partial_t^- f$ and $\partial_x^- f$ exist, are left-continuous and locally bounded.
\par
(3) $\partial_x^- f$ is of locally bounded variation in $\mathbb{R}_+\times \mathbb{R}$ and $\partial_x^- f(0,\cdot)$ is of locally bounded variation in $\mathbb{R}$.
\par
Then we have the following extension meyer-tanaka formula to time-dependent functions:

\begin{align*}
f(t,X_t)&=f(0,X_0)+\int_0^t \partial_t^- f(s,X_s) ds+\int_0^t \partial_x^- f(s,X_s) dX_s\\\nonumber
&+\int_{\mathbb{R}}L^X(t,y) d_y \partial_x^- f(t,y)-\int_{\mathbb{R}}\int_0^t L^X(s,y) d_{s,y} \partial_x^- f(s,y),
\end{align*}
where the local time is defined by the limit in probability:
\begin{equation}
L^X(t,y)=\frac{1}{2}\int_0^t \delta(X_s-y)d\langle{X}\rangle_s
\end{equation}
where $\delta(x)=0,~(x\neq0)$ and $\int_{-\infty}^{\infty} \delta(x)dx=1$. The notations $d_y$ and $d_{s,y}$ mean integration with respect to the $y$ variable and the $(s,y)$ variables, respectively.
\end{lemma}

According to the Theorem 2.1 of \citet{MR2884811} and Theorem 3.1 of \citet{MR2122800}, we have the following result.

\begin{theorem}\label{theorem3.1}
Let $0<x<b$. Then the joint LT of $(X_{\tau},\tau)$ is given by
\begin{equation}\label{e:M}
\Psi(\alpha,\theta;x)=\omega_1(\lambda+\theta;b,x)+e^{-\alpha b}\omega_2(\lambda+\theta;b,x)+\mathbf{E}_x[e^{-\theta T_1}M(\theta,X_{T_1},Y_1)\mathbf{1}_{\{\tau\geq T_1\}}],
\end{equation}
where the function $M$ is defined by
\begin{equation*}
M(\theta,x,y)=\omega_1(\lambda+\theta;b+y,x)+e^{-\alpha(b+y)}\omega_2(\lambda+\theta;b+y,x), ~y>0\;,
\end{equation*}
the functions  $\omega_1$ and $\omega_2$ are defined in \eqref{omega1} and \eqref{omega2}. The last term on the right side  of \eqref{e:M} can be determined by the following Proposition \ref{prop:1} and Proposition \ref{prop:2}.
\end{theorem}

Next, we are ready to prove the Theorem \ref{theorem3.1}. By using the Lemma \ref{le:4}, we can obtain the LTs evaluated at the exponential random time $T_1$, i.e., $\Phi_1(\alpha,\theta;x)$ and $\Phi_2(\alpha,\theta;x)$.

\begin{prop}\label{prop:1}
Let $x \in \mathbb{R}$ and $\alpha, \theta>0$. Then
\begin{align}
\Phi_{1}(\alpha,\theta;x)&=\frac{\lambda(e^{-\alpha x}\mathbf{1}_{\{x<c\}}+e^{-\alpha c}\mathbf{1}_{\{x\geq c\}})-(\lambda+\theta)e^{-\alpha c}g(\theta;x)}{\lambda+\theta+\alpha\mu_1-\frac{1}{2}\alpha^2},\label{e:00}\\
\Phi_{2}(\alpha,\theta;x)&=\frac{\lambda(e^{-\alpha x}\mathbf{1}_{\{x\geq c\}}-e^{-\alpha c}\mathbf{1}_{\{x\geq c\}})+(\lambda+\theta)e^{-\alpha c}g(\theta;x)}{\lambda+\theta+\alpha\mu_2-\frac{1}{2}\alpha^2},\label{e:01}
\end{align}
where
\begin{equation}\label{e:g}
g(\theta;x)=\mathbf{E}_x \left[e^{-\theta T_1}\int_c^{\infty}p(T_1;x,y)dy \right]
\end{equation}
and $p(t;x,y)$ is the transition density function for the BM with broken drift (i.e. stochastic process $X$) that can be obtained in \eqref{e:p}.
\end{prop}

\begin{proof}
By using It\^o's formula to $e^{-\alpha X_t-\theta t}$, we have
\begin{equation*}
M_t=e^{-\alpha X_t-\theta t}-e^{-\alpha X_0}+(\theta-\frac{1}{2}\alpha^2)\int_0^t e^{-\alpha X_s-\theta s}ds+\alpha \int_0^t \mu(X_s)e^{-\alpha X_s-\theta s}ds
\end{equation*}
is a martingale. Now, applying the optional stopping time theorem to $M$ and $T_1$, we obtain
\begin{equation}\label{e:zhu1}
\Phi_{1}(\alpha,\theta;x)+\Phi_{2}(\alpha,\theta;x)-e^{-\alpha x}+(\theta-\frac{1}{2}\alpha^2)\mathbf{E}_x \left[\int_0^{T_1} e^{-\alpha X_s-\theta s}ds\right]+\alpha \mathbf{E}_x \left[\int_0^{T_1} \mu(X_s)e^{-\alpha X_s-\theta s}ds\right]=0.
\end{equation}
Since the exponential random variable $T_1$ is independent of the process $X$, we have
\begin{equation*}
\mathbf{E}_x \left[\int_0^{T_1} e^{-\alpha X_s-\theta s}ds \right]=\frac{1}{\lambda}(\Phi_{1}(\alpha,\theta;x)+\Phi_{2}(\alpha,\theta;x))
\end{equation*}
and
\begin{equation*}
\mathbf{E}_x\left[\int_0^{T_1} \mu(X_s)e^{-\alpha X_s-\theta s}ds\right]=\frac{1}{\lambda}\mathbf{E}_x\left[ \mu(X_{T_1})e^{-\alpha X_{T_1}-\theta T_1}\right].
\end{equation*}
Thus, we yield the following equation
\begin{equation}\label{e:sx2}
\frac{\lambda+\theta+\alpha\mu_1-\frac{1}{2}\alpha^2}{\lambda}\Phi_{1}(\alpha,\theta;x)=e^{-\alpha x}-\frac{\lambda+\theta+\alpha\mu_2-\frac{1}{2}\alpha^2}{\lambda}\Phi_{2}(\alpha,\theta;x).
\end{equation}
Define the function $h(t,x)$ by $h(t,x)=e^{-\alpha x-\theta t}\mathbf{1}_{\{x<c\}}+e^{-\alpha c-\theta t}\mathbf{1}_{\{x\geq c\}}$. Obviously, $h(t,x)$ satisfies the conditions in Lemma \ref{le:4}, as a consequence, we yield that
\begin{align*}
N_t&=e^{-\alpha X_t-\theta t}\mathbf{1}_{\{X_t<c\}}+e^{-\alpha c-\theta t}\mathbf{1}_{\{X_t\geq c\}}-e^{-\alpha x}\mathbf{1}_{\{x<c\}}-e^{-\alpha c}\mathbf{1}_{\{x\geq c\}}\\\nonumber
&+(\alpha\mu_1+\theta)\int_0^t e^{-\alpha X_s-\theta s}\mathbf{1}_{\{X_s\leq c\}}ds+\theta\int_0^t e^{-\alpha c-\theta s}\mathbf{1}_{\{X_s>c\}}ds-\\\nonumber
&\alpha^2\int_{\mathbb{R}}L^X(t,y) e^{-\alpha y-\theta t}\mathbf{1}_{\{y\leq c\}}dy-\alpha^2\theta\int_{\mathbb{R}}\int_0^t L^X(s,y) e^{-\alpha y-\theta s}\mathbf{1}_{\{y\leq c\}}dsdy,
\end{align*}
is a martingale. By using the occupation times formula, we obtain
\begin{equation*}
\int_{\mathbb{R}}L^X(t,y) e^{-\alpha y-\theta t}\mathbf{1}_{\{y\leq c\}}dy=\frac{1}{2}e^{-\theta t}\int_0^t e^{-\alpha X_s}\mathbf{1}_{\{X_s\leq c\}}ds
\end{equation*}
and similarly
\begin{equation*}
\int_{\mathbb{R}}\int_0^t L^X(s,y) e^{-\alpha y-\theta s}\mathbf{1}_{(y\leq c)}dsdy=\frac{1}{2}\int_0^t \int_0^s e^{-\alpha X_u}\mathbf{1}_{\{X_u\leq c\}}du e^{-\theta s}ds.
\end{equation*}
Then, applying the optional stopping time theorem to $N$ and $T_1$, we obtain
\begin{align}\label{e:sx}
&\mathbf{E}_x\left[e^{-\alpha X_{T_1}-\theta T_1}\mathbf{1}_{\{X_{T_1}<c\}}\right]+\mathbf{E}_x\left[e^{-\alpha c-\theta T_1}\mathbf{1}_{\{X_{T_1}\geq c\}}\right]=e^{-\alpha x}\mathbf{1}_{\{x<c\}}+e^{-\alpha c}\mathbf{1}_{\{x\geq c\}}-\frac{\alpha\mu_1+\theta}{\lambda}\nonumber\\
&\mathbf{E}_x\left[ e^{-\alpha X_{T_1}-\theta T_1}\mathbf{1}_{\{X_{T_1}< c\}}\right]-\frac{\theta e^{-\alpha c}}{\lambda}\mathbf{E}_x\left[e^{-\theta T_1}\mathbf{1}_{\{X_{T_1}\geq c\}}\right]+\frac{\alpha^2}{2}\frac{\lambda+\theta}{\lambda}\mathbf{E}_x\left[ e^{-\theta T_1}\int_0^{T_1} e^{-\alpha X_s} \mathbf{1}_{\{X_s\leq c\}}ds\right].
\end{align}
Note that we use two facts in the last equality. One is that $T_1$ is an exponential random variable which is independent of process $X$ and the other is $\mathbf{E}_x\left[e^{-\alpha X_{T_1}-\theta T_1}\mathbf{1}_{\{X_{T_1}= c\}}\right]=0$ and $\mathbf{E}_x\left[e^{-\theta T_1}\mathbf{1}_{\{X_{T_1}= c\}}\right]=0$. In addition, because of the definition of the function $g(\theta;x)$, i.e., \eqref{e:g}, we have
\begin{equation*}
g(\theta;x)=\mathbf{E}_x\left[e^{-\theta T_1}\mathbf{1}_{\{X_{T_1}\geq c\}}\right].
\end{equation*}
Moreover,
\begin{align*}
\mathbf{E}_x\left[ e^{-\theta T_1}\int_0^{T_1} e^{-\alpha X_s} \mathbf{1}_{\{X_s\leq c\}}ds\right]=&\mathbf{E}_x\left[\int_0^{T_1} \int_0^t e^{-\alpha X_s}\mathbf{1}_{\{X_s\leq c\}}dsde^{-\theta t}+\int_0^{T_1}e^{-\theta t} e^{-\alpha X_t}\mathbf{1}_{\{X_t\leq c\}}dt\right]\\\nonumber
=&\mathbf{E}_x\left[ -\int_0^{T_1} \theta e^{-\theta t}\int_0^t e^{-\alpha X_s}\mathbf{1}_{\{X_s\leq c\}}dsdt\right]+\frac{1}{\lambda}\Phi_{1}(\alpha,\theta;x)\\\nonumber
=&-\frac{\theta}{\lambda}\mathbf{E}_x\left[ e^{-\theta T_1}\int_0^{T_1} e^{-\alpha X_s} \mathbf{1}_{\{X_s\leq c\}}ds\right]+\frac{1}{\lambda}\Phi_{1}(\alpha,\theta;x).
\end{align*}
Hence we have
\begin{equation}\label{e:sx1}
\mathbf{E}_x\left[ e^{-\theta T_1}\int_0^{T_1} e^{-\alpha X_s} \mathbf{1}_{\{X_s\leq c\}}ds\right]=\frac{1}{\lambda+\theta}\Phi_{1}(\alpha,\theta;x).
\end{equation}
Combining \eqref{e:sx} and \eqref{e:sx1}, we arrive at \eqref{e:00}. Finally we can derive \eqref{e:01} from \eqref{e:00} and \eqref{e:sx2}.
\end{proof}

\begin{remark}
If there is no broken phenomenon in drift term, the equation \eqref{e:zhu1} will be same as the equation ($3.17$) in \citet{MR2122800}. Because of the broken phenomenon, the computations used in \citet{MR2122800} cannot be applied to this case. We need to construct a function by using time-dependent tanaka formula to deal with the problem.

\end{remark}

\begin{prop}\label{prop:2}
For $0<x<b$, we can yield that
\begin{align*}
\Phi_3(\alpha,\theta;x)=\Phi_{1}(\alpha,\theta;0)\omega_1(\lambda+\theta;b,x)+\Phi_{1}(\alpha,\theta;b)\omega_2(\lambda+\theta;b,x),\\
\Phi_4(\alpha,\theta;x)=\Phi_{2}(\alpha,\theta;0)\omega_1(\lambda+\theta;b,x)+\Phi_{2}(\alpha,\theta;b)\omega_2(\lambda+\theta;b,x).
\end{align*}
where $\omega_1(\lambda+\theta;b,x)$ and $\omega_2(\lambda+\theta;b,x)$ are given in \eqref{omega1} and \eqref{omega2} respectively, and $\Phi_{1}(\alpha,\theta;0)$, $\Phi_{2}(\alpha,\theta;0)$ are given in Proposition \ref{prop:1}.
\end{prop}

\begin{proof}
Firstly, we divide $\Phi_{1}(\alpha,\theta;x)$ into two parts.
\begin{equation*}
\Phi_3(\alpha,\theta;x)=\mathbf{E}_x\left[e^{-\alpha X_{T_1}-\theta T_1}\mathbf{1}_{\{X_{T_1<c}\}}\mathbf{1}_{\{\tau<T_1,~X_{\tau}=0\}}\right]+\mathbf{E}_x\left[e^{-\alpha X_{T_1}-\theta T_1}\mathbf{1}_{\{X_{T_1<c}\}}\mathbf{1}_{\{\tau<T_1,~X_{\tau}=b\}}\right].
\end{equation*}
By using the strong Markov property of the process $X$, the memoryless property of the exponential distribution and the two equalities $T_1=\tau+T_1\circ \zeta_{\tau}$ and $X_{T_1}=X_{T_1\circ \zeta_{\tau}}$ on the event $\{\omega:\tau(\omega)<T_1(\omega)\}$, where $\zeta_{\cdot}$ is the shift operator, we have that
\begin{align*}
\mathbf{E}_x\left[e^{-\alpha X_{T_1}-\theta T_1}\mathbf{1}_{\{X_{T_1<c}\}}\mathbf{1}_{\{\tau<T_1,~X_{\tau}=0\}}\right]=&\mathbf{E}_x\left[\mathbf{E}_x(e^{-\alpha X_{T_1}-\theta T_1}\mathbf{1}_{\{X_{T_1<c}\}}\mathbf{1}_{\{\tau<T_1,~X_{\tau}=0\}}|\mathcal{F}_{\tau})\right]\\
=&\mathbf{E}_x\left[\mathbf{1}_{\{\tau<T_1,~X_{\tau}=0\}}\mathbf{E}_x(e^{-\alpha X_{T_1}-\theta T_1}\mathbf{1}_{\{X_{T_1<c}\}}|\mathcal{F}_{\tau})\right]\\
=&\mathbf{E}_x\left[e^{-\theta\tau}\mathbf{1}_{\{\tau<T_1,~X_{\tau}=0\}}\mathbf{E}_x(e^{-\alpha X_{T_1\circ \zeta_{\tau}}-\theta T_1\circ \zeta_{\tau}}\mathbf{1}_{\{X_{T_1\circ \zeta_{\tau}<c}\}}|\mathcal{F}_{\tau})\right]\\
=&\Phi_{1}(\alpha,\theta;0)\mathbf{E}_x\left[e^{-\theta\tau}\mathbf{1}_{\{\tau<T_1,~X_{\tau}=0\}}\right]\\
=&\Phi_{1}(\alpha,\theta;0)\mathbf{E}_x\left[e^{-\theta(R_0\wedge R_b)}\mathbf{1}_{\{R_0<R_b\}}P(T_1>R_0\wedge R_b|X)\right]\\
=&\Phi_{1}(\alpha,\theta;0)\mathbf{E}_x\left[e^{-(\theta+\lambda)R_0}\mathbf{1}_{\{R_0<R_b\}}\right]
\end{align*}
Recall the definition of the function $\omega_1(\theta;b,x)$ in \eqref{e:w1}, we obtain
\begin{equation}\label{e:sx3}
\mathbf{E}_x\left[e^{-\alpha X_{T_1}-\theta T_1}\mathbf{1}_{\{X_{T_1<c}\}}\mathbf{1}_{\{\tau<T_1,~X_{\tau}=0\}}\right]=\Phi_{1}(\alpha,\theta;0)\omega_1(\lambda+\theta;b,x).
\end{equation}
Similarly,
\begin{equation}\label{e:sx4}
\mathbf{E}_x\left[e^{-\alpha X_{T_1}-\theta T_1}\mathbf{1}_{\{X_{T_1<c}\}}\mathbf{1}_{\{\tau<T_1,~X_{\tau}=b\}}\right]=\Phi_{1}(\alpha,\theta;b)\omega_2(\lambda+\theta;b,x).
\end{equation}
Combining \eqref{e:sx3} and \eqref{e:sx4}, we arrive at the expression of $\Phi_3(\alpha,\theta;x)$. Analogusly, we can obtain the expression of $\Phi_4(\alpha,\theta;x)$. This completes the proof.
\end{proof}

\subsection{Hitting problem of RBM with broken drift}
\par
In this subsection, we shall study the first passage time $\tilde{\tau}$ defined in \eqref{e:tau2} and our main aim is to determine the joint Laplace transform:
\begin{equation*}
\tilde{\Psi}(\alpha,\theta;x)=\mathbf{E}_x\left[e^{-\alpha \tilde{X_{\tilde{\tau}}}-\theta\tilde{\tau}}\right],
\end{equation*}
and similarly as the last section, we need to compute the following four Laplace transforms,
\begin{align*}
\tilde{\Phi}_1(\alpha,\theta;x)&=\mathbf{E}_x\left[e^{-\alpha \tilde{X}_{T_1}-\theta T_1}\mathbf{1}_{\{X_{T_1<c}\}}\right],\\
\tilde{\Phi}_2(\alpha,\theta;x)&=\mathbf{E}_x\left[e^{-\alpha \tilde{X}_{T_1}-\theta T_1}\mathbf{1}_{\{X_{T_1\geq c}\}}\right],\\
\tilde{\Phi}_3(\alpha,\theta;x)&=\mathbf{E}_x\left[e^{-\alpha \tilde{X}_{T_1}-\theta T_1}\mathbf{1}_{\{X_{T_1<c}\}}\mathbf{1}_{\{\tilde{\tau}<T_1\}}\right],\\
\tilde{\Phi}_4(\alpha,\theta;x)&=\mathbf{E}_x\left[e^{-\alpha \tilde{X}_{T_1}-\theta T_1}\mathbf{1}_{\{X_{T_1\geq c}\}}\mathbf{1}_{\{\tilde{\tau}<T_1\}}\right].
\end{align*}

\begin{theorem}
Let $0<x<b$. Then the joint LT of $(\tilde{X}_{\tilde{\tau}},\tilde{\tau})$ is given by
\begin{equation*}
\tilde{\Psi}(\alpha,\theta;x)=e^{-\alpha b}\frac{\tilde{\psi}_{\theta+\lambda}(x)}{\tilde{\psi}_{\theta+\lambda}(b)}+e^{-\alpha b}\mathbf{E}_x\left[e^{-\theta T_1-\alpha Y_1}\frac{\tilde{\psi}_{\theta}(x)}{\tilde{\psi}_{\theta}(b+Y_1)}\mathbf{1}_{\{\tilde{\tau}\geq T_1\}}\right],
\end{equation*}
where  the last term on the right side can be determined by the following Proposition \ref{prop:3} and Proposition \eqref{prop:4} and  $\tilde{\psi}_{\theta+\lambda}(x)$ is given in \eqref{e:tilde(psai)}.
\end{theorem}

\begin{prop}\label{prop:3}
Let $x \in \mathbb{R}$ and $\alpha, \theta>0$. Then
\begin{align}
\tilde{\Phi}_{1}(\alpha,\theta;x)&=\frac{\lambda(e^{-\alpha x}\mathbf{1}_{\{x<c\}}+e^{-\alpha c}\mathbf{1}_{\{x\geq c\}})-(\lambda+\theta)e^{-\alpha c}g_0(\theta;x)-\lambda\alpha g_1(\theta;x)}{\lambda+\theta+\alpha\mu_1-\frac{1}{2}\alpha^2},\\
\tilde{\Phi}_{2}(\alpha,\theta;x)&=\frac{\lambda(e^{-\alpha x}\mathbf{1}_{\{x\geq c\}}-e^{-\alpha c}\mathbf{1}_{\{x\geq c\}})+(\lambda+\theta)e^{-\alpha c}g_0(\theta;x)}{\lambda+\theta+\alpha\mu_2-\frac{1}{2}\alpha^2},\label{e:tildephi01}
\end{align}
where $g_0(\theta;x)=\mathbf{E}_x\left[e^{-\theta T_1}\int_c^{\infty}\tilde{p}(T_1;x,y)dy\right]$, $g_1(\theta;x)=\frac{1}{2\lambda}\mathbf{E}_x\left[e^{-\theta T_1}\tilde{p}(T_1;x,0)\right]$ and $\tilde{p}(t;x,y)$ is the transition density function for the RBM with broken drift (i.e. stochastic process $\tilde{X}$) that can be obtained in \eqref{e:tp}.
\end{prop}

\begin{proof}
By applying It\^o formula to $e^{-\alpha \tilde{X}_t-\theta t}$, we have
\begin{align*}
\tilde{M}_t=&e^{-\alpha \tilde{X}_t-\theta t}-e^{-\alpha \tilde{X}_0}+(\theta-\frac{1}{2}\alpha^2)\int_0^t e^{-\alpha \tilde{X}_s-\theta s}ds+\alpha \int_0^t \mu(\tilde{X}_s)e^{-\alpha \tilde{X}_s-\theta s}ds+\alpha \int_0^t e^{-\alpha \tilde{X}_s-\theta s}dL_s
\end{align*}
is a martingale. Then, applying the property of the local time $L_t$ in \eqref{e:Lt}, we obtain
\begin{align*}
&\tilde{\Phi}_{1}(\alpha,\theta;x)+\tilde{\Phi}_{2}(\alpha,\theta;x)-e^{-\alpha x}+(\theta-\frac{1}{2}\alpha^2)\mathbf{E}_x\left[\int_0^{T_1} e^{-\alpha \tilde{X}_s-\theta s}ds\right]\\\nonumber
&+\alpha \mathbf{E}_x\left[\int_0^{T_1} \mu(\tilde{X}_s)e^{-\alpha \tilde{X}_s-\theta s}ds\right]+\alpha\mathbf{E}_x\left[\int_0^{T_1} e^{-\theta s}dL_s\right]=0.
\end{align*}
Similarly as the proof of the Proposition \ref{prop:1}, we yield
\begin{equation}\label{e:tilde1}
\frac{\lambda+\theta+\alpha\mu_1-\frac{1}{2}\alpha^2}{\lambda}\tilde{\Phi}_{1}(\alpha,\theta;x)=e^{-\alpha x}-\alpha g_1(\theta;x)-\frac{\lambda+\theta+\alpha\mu_2-\frac{1}{2}\alpha^2}{\lambda}\tilde{\Phi}_{2}(\alpha,\theta;x),
\end{equation}
where the function $g_1(\theta;x)$ is defined by
\begin{equation*}
g_1(\theta;x)=\mathbf{E}_x\left[\int_0^{T_1} e^{-\theta s}dL_s\right].
\end{equation*}
According to the equation (3.9) of \citet{MR2884811}, the function $g_1(\theta,x)$ should be formulated by
\begin{equation*}
g_1(\theta;x)=\frac{1}{2\lambda}\mathbf{E}_x\left[e^{-\theta T_1}\tilde{p}(T_1;x,0)\right].
\end{equation*}
Since $h(t,x)=e^{-\alpha x-\theta t}\mathbf{1}_{\{x<c\}}+e^{-\alpha c-\theta t}\mathbf{1}_{\{x\geq c\}}$ and $h(t,x)$ satisfies the conditions in Lemma \ref{le:4}, as a consequence, we get that
\begin{align*}
\tilde{N}_t&=e^{-\alpha \tilde{X}_t-\theta t}\mathbf{1}_{\{\tilde{X}_t<c\}}+e^{-\alpha c-\theta t}\mathbf{1}_{\{\tilde{X}_t\geq c\}}-e^{-\alpha x}\mathbf{1}_{\{x<c\}}-e^{-\alpha c}\mathbf{1}_{\{x\geq c\}}+(\alpha\mu_1+\theta)\\\nonumber
&\int_0^t e^{-\alpha \tilde{X}_s-\theta s}\mathbf{1}_{\{\tilde{X}_s\leq c\}}ds+\theta\int_0^t e^{-\alpha c-\theta s}\mathbf{1}_{\{\tilde{X}_s>c\}}ds-\alpha^2\int_{\mathbb{R}}L^{\tilde{X}}(t,y) e^{-\alpha y-\theta t}\mathbf{1}_{\{y\leq c\}}dy\\\nonumber
&-\alpha^2\theta\int_{\mathbb{R}}\int_0^t L^{\tilde{X}}(s,y) e^{-\alpha y-\theta s}\mathbf{1}_{\{y\leq c\}}dsdy+\alpha \int_0^t e^{-\alpha \tilde{X}_s-\theta s}\mathbf{1}_{\{\tilde{X}_s\leq c\}}dL_s,
\end{align*}
is a martingale. It is similar to the proof of the Proposition \ref{prop:1} to obtain
\begin{align}\label{e:tilde2}
\tilde{\Phi}_{1}(\alpha,\theta;x)&=\frac{\lambda(e^{-\alpha x}\mathbf{1}_{\{x<c\}}+e^{-\alpha c}\mathbf{1}_{\{x\geq c\}})-(\lambda+\theta)e^{-\alpha c}g_0(\theta;x)-\lambda\alpha g_1(\theta;x)}{\lambda+\theta+\alpha\mu_1-\frac{1}{2}\alpha^2}.
\end{align}

Combining \eqref{e:tilde1} and \eqref{e:tilde2}, we arrive at \eqref{e:tildephi01}.

\end{proof}

\begin{prop}\label{prop:4}
For $0<x<b$, we have that
\begin{align*}
\tilde{\Phi}_3(\alpha,\theta;x)=\tilde{\Phi}_{1}(\alpha,\theta;b)\frac{\tilde{\psi}_{\theta+\lambda}(x)}{\tilde{\psi}_{\theta+\lambda}(b)},\\
\tilde{\Phi}_4(\alpha,\theta;x)=\tilde{\Phi}_{2}(\alpha,\theta;b)\frac{\tilde{\psi}_{\theta+\lambda}(x)}{\tilde{\psi}_{\theta+\lambda}(b)},
\end{align*}
where $\tilde{\psi}_{\theta+\lambda}(x)$ is given in \eqref{e:tilde(psai)}, and $\tilde{\Phi}_{0}(\alpha,\theta;0)$ is given in Proposition \ref{prop:3}.
\end{prop}

\begin{proof}
The proof is similar to that of Proposition \ref{prop:2}.
\end{proof}

\section*{Acknowledgement}
The authors are indebted to Professor Yongjin Wang for introducing them the topic of this research and many suggestions. This work is supported by the National Natural science Foundation of China (No. 11631004, 71532001).

\bibliography{bi}


\end{document}